\documentclass[11pt]{article}
\usepackage{amsfonts, amsmath, amssymb, amscd} 
\usepackage{amsthm}
\usepackage{amscd}
\usepackage{fancyhdr}
\usepackage{setspace}
\usepackage[margin=1in]{geometry}
\usepackage{cite}

\newtheorem{definition}{Definition}[section]
\newtheorem{theorem}[definition]{Theorem}
\newtheorem{lemma}[definition]{Lemma}
\newtheorem{corollary}[definition]{Corollary}
\newtheorem{remark}[definition]{Remark}

\newtheorem{proposition}[definition]{Proposition}

\fancypagestyle{plain}{ 
\fancyhf{}
\fancyfoot[R]{\thepage}
}

\title{Tridiagonal pairs of $q$-Serre type and their linear perturbations}
\author{Aayush Karan}
\date{Harvard University, \\ 86 Brattle Street, Cambridge, MA, USA 02138 \\ {akaran@college.harvard.edu}} \vspace{\baselineskip} \vspace{\baselineskip} \vspace{\baselineskip} 
\begin{document}
\maketitle
\vspace{\baselineskip} 

\begin{abstract}
\noindent
A tridiagonal pair is an ordered pair of diagonalizable linear maps on a nonzero finite-dimensional vector space, that each act on the eigenspaces of the other in a block-tridiagonal fashion. We consider a tridiagonal pair $(A, A^*)$ of $q$-Serre type; for such a pair the maps $A$ and $A^*$ satisfy the $q$-Serre relations. There is a linear map $K$ in the literature that is used to describe how $A$ and $A^*$ are related. We investigate a pair of linear maps $B=A$ and $B^* = tA^* + (1-t)K$, where $t$ is any scalar. Our goal is to find a necessary and sufficient condition on $t$ for the pair $(B, B^*)$ to be a tridiagonal pair. We show that $(B, B^*)$ is a tridiagonal pair if and only if $t \neq 0$ and $P \bigl(  t(q-q^{-1})^{-2}  \bigr)\not=0$, where $P$ is a certain polynomial attached to $(A, A^*)$ called the Drinfel'd polynomial.
\end{abstract}

\vspace{\baselineskip}
\begin{flushleft}

\bigskip

\end{flushleft}

\newpage

\pagenumbering{arabic}
\section{Introduction}
In this paper we consider a linear algebraic object known
as a tridiagonal pair \cite{pqassoc}. Roughly speaking, a tridiagonal pair is an ordered pair of diagonalizable linear maps $(A, A^*)$ on a nonzero finite-dimensional vector space, that each act on the eigenspaces of the other in a block-tridiagonal fashion. We consider a tridiagonal pair $(A, A^*)$ of $q$-Serre type; for such a pair the $A$ and $A^*$ satisfy the $q$-Serre relations as shown in (\ref{qSer1}) and (\ref{qSer2}) below. There is a linear map $K$ in the literature (see \cite[Section~1.1]{yetanotk} and \cite{k, anotk, kk}) that is used to describe how $A$ and $A^*$ are
related. We investigate a pair of linear maps $(B, B^*)$ of the form \begin{equation}
    \nonumber
    B = A, \qquad \qquad B^* = tA^* + (1-t)K,
\end{equation}
where $t$ is any scalar. Our goal is to find a necessary and sufficient condition on $t$ such that $(B, B^*)$ is a tridiagonal pair. To reach this goal, we associate with $(A, A^*)$ a polynomial $P(x)$ called the Drinfel'd polynomial \cite{drff, drinfeld, anotdrf}. The degree of $P$ plus one is equal to the number of eigenspaces of $A$ and the number of eigenspaces of $A^*$. We show that the pair $(B, B^*)$ is a tridiagonal pair if and only if both \begin{equation}\nonumber t \neq 0, \qquad \qquad P\left(\frac{t}{(q - q^{-1})^2}\right) \neq 0.\end{equation} Our main result is Theorem \ref{maintheorem}.

\bigskip
\noindent
The paper is organized as follows. In Section \ref{one} we make some basic definitions and set some
notation. In Section \ref{pssp} we introduce the notion of a parallel system. In Section \ref{tdsyssss} we recall the notion of a tridiagonal system. In Section \ref{spltdecompo} we recall the split decomposition. In Section \ref{mtdsysysys} we recall the notion of a mock tridiagonal system. In Section \ref{linpert} we introduce the notion of a $t$-linear perturbation of a tridiagonal pair and prove a number of results about it. In Section \ref{dfp} we define the Drinfel'd polynomial and use it to prove our main result.

\section{Preliminaries}\label{one}

In this section we give some basic definitions and set our notation. Let $\mathcal{K}$ denote an algebraically closed field. Throughout this paper, a scalar will refer to an element of $\mathcal{K}$ and every vector space we mention will be understood to be over $\mathcal{K}$. Let $V$ denote a nonzero finite-dimensional vector space. For a linear map $A: V \to V$, an \textit{eigenspace} refers to a nonzero subspace $\{v \in V : Av = \lambda v\}$ for some scalar $\lambda$. This scalar is called the \textit{eigenvalue} for the given eigenspace. The map $A$ is said to be \textit{diagonalizable} whenever $V$ is spanned by the eigenspaces of $A$.

\begin{definition}\rm{(See \cite[Definition~1.1]{pqassoc})}{\label{TD}}
\rm
A \textit{tridiagonal pair} on $V$ is an ordered pair $(A, A^*)$ of linear maps $A:V \to V$ and $A^*:V \to V$ that satisfy the following conditions:

\begin{itemize}
    \item[\rm (i)] $A$ and $A^*$ are diagonalizable.
    
    \item[\rm (ii)] There exists an ordering $\{V_i\}^d_{i=0}$ of the eigenspaces of $A$ such that 
    \begin{equation}
    \nonumber
    A^*V_i \subseteq V_{i-1} + V_{i} + V_{i+1},\end{equation}
    where $0 \leq i \leq d$ and $V_{-1} = V_{d+1} = 0.$
    
    \item[\rm (iii)] There exists an ordering $\{V^*_i\}^{\delta}_{i=0}$ of the eigenspaces of $A^*$ such that 
    \begin{equation}\nonumber
AV^*_i \subseteq V^*_{i-1} + V^*_{i} + V^*_{i+1},\end{equation}
    where $0 \leq i \leq {\delta}$ and $V^*_{-1} = V^*_{{\delta}+1} = 0.$
    
    \item[\rm (iv)] If a subspace $W$ of $V$ is such that $AW \subseteq W$ and $A^*W \subseteq W$, then either $W = 0$ or $W = V$.
\end{itemize}
The above tridiagonal pair is said to be \textit{over} $\mathcal{K}.$

\end{definition}
\begin{flushleft}
With reference to the above definition, it is known that $d = \delta$ (see \cite[Lemma $4.5$]{pqassoc}), so $A$ and $A^*$ have the same number of eigenspaces.  

\bigskip

Note that if $(A, A^*)$ is a tridiagonal pair on $V$, then so is $(A^*, A).$
\begin{remark}\rm
For a tridiagonal pair $(A, A^*)$, call an ordering $\{V_i\}_{i=0}^d$ of the eigenspaces of $A$ \textit{standard} if it satisfies the conditions of Definition \ref{TD}$\rm(ii)$. Observe that a standard ordering is not unique, as the ordering $\{V_{d-i}\}^d_{i=0}$ is also standard. Since $(A^*, A)$ is also a tridiagonal pair, a similar discussion applies to an ordering $\{V^*_i\}_{i=0}^d$ of the eigenspaces of $A^*$.
\end{remark}

For a given tridiagonal pair $(A, A^*)$, observe that a standard ordering of the eigenspaces of $A$ gives an ordering $\{\theta_i\}_{i=0}^d$ of the eigenvalues of $A$ known as an \textit{eigenvalue sequence} of $(A, A^*).$ Looking instead at the tridiagonal pair $(A^*, A)$, we obtain the \textit{dual eigenvalue sequence} $\{\theta^*_i\}_{i=0}^d.$ Note that since a standard ordering is not unique, neither are the eigenvalue and dual eigenvalue sequences. Indeed, $\{\theta_{d-i}\}_{i=0}^d$ is also an eigenvalue sequence, while $\{\theta^*_{d-i}\}_{i=0}^d$ is also a dual eigenvalue sequence.

\end{flushleft}

\begin{flushleft}
The tridiagonal conditions impose a great deal of structure on the linear maps $A$ and $A^*$; for instance, the following relations must always be satisfied.

\end{flushleft}

\begin{theorem} \rm{(See \cite[Theorem $10.1$]{pqassoc})}\label{TDrel}
\em
Let $(A, A^*)$ denote a tridiagonal pair over  $\mathcal{K}$. Then there exist scalars $\beta, \gamma, \gamma^*, \rho, \rho^*$ such that 
\begin{align}
    & [A, A^2A^* - \beta AA^*A + A^*A^2 - \gamma(AA^* + A^*A) - \rho A^*] = 0 \label{TDrel1},
    \\
    & [A^*, A^{*2}A - \beta A^*AA^* + AA^{*2} - \gamma^*(AA^* + A^*A) - \rho^* A] = 0 \label{TDrel2}
\end{align}
where $[X, Y] = XY - YX$.
\end{theorem}
\begin{remark}\rm{(See \cite{physics}) }\label{q}\rm The relations (\ref{TDrel1}) and (\ref{TDrel2}) are relevant in physics, with certain cases appearing in quantum integrable models and exactly solvable systems in statistical mechanics.
\end{remark}

\begin{flushleft}
For the rest of the paper, we fix a nonzero scalar $q$ that is not a root of unity. Our primary focus is a special case of (\ref{TDrel1}) and (\ref{TDrel2}) known as the $q$-Serre relations. This special case is described as follows. Setting  $\gamma = \gamma^* = \rho = \rho^* = 0$ and $\beta = q^2 + q^{-2}$, the relations of Theorem \ref{TDrel} become
\begin{align}
    & A^3A^* - [3]_qA^2A^*A + [3]_qAA^*A^2 - A^*A^3 = 0, \label{qSer1}
    \\
    & A^{*3}A - [3]_qA^{*2}AA^* + [3]_qA^*AA^{*2} - AA^{*3} = 0, \label{qSer2}
\end{align}
where $[i]_q = \frac{q^i-q^{-i}}{q-q^{-1}}$ for any $i \in \mathbb{Z}.$ 

\begin{definition}
\rm
We say a tridiagonal pair $(A, A^*)$ has \textit{$q$-Serre type} if it satisfies relations (\ref{qSer1}) and (\ref{qSer2}).
\end{definition}

\noindent The next result uses (\ref{qSer1}) and (\ref{qSer2}) to describe the eigenvalue sequences and dual eigenvalue sequences of a tridiagonal pair of $q$-Serre type.
\begin{proposition} \rm{(See \cite[Lemma $4.8$]{qSerre})}\label{seq}
\em
Let $(A, A^*)$ be a tridiagonal pair over $\mathcal{K}$. Then the following are equivalent:
\begin{itemize}
    \item[\rm (i)] $(A, A^*)$ satisfies the $q$-Serre relations.
    \item[\rm (ii)] There exists an eigenvalue sequence $\{\theta_{i}\}_{i=0}^d$ for $(A, A^*)$ and a dual eigenvalue sequence $\{\theta^*_{i}\}_{i=0}^d$ for $(A, A^*)$ such that $\theta_i = q^{2i}\theta_0$ and $\theta^*_i = q^{2i}\theta^*_0$, for $0 \leq i \leq d$.
\end{itemize}
\end{proposition}
\end{flushleft}

\section{Parallel Systems and Split Sequences}\label{pssp}
In this section we present some basic linear algebraic constructions and corresponding results. We begin with some notation. Let $x$ denote an indeterminate, and let $\mathcal{K}[x]$ denote the algebra of polynomials in $x$ that have all coefficients in $\mathcal{K}$.
\begin{definition}
\rm
Given scalars $\{\theta_i\}_{i=0}^d$ and $\{\theta^*_i\}_{i=0}^d$, we define some polynomials in $\mathcal{K}[x]$ as follows. For $0 \leq i \leq d$,
\begin{align}
& \eta_i = (x - \theta_d)(x - \theta_{d-1})\cdots(x - \theta_{d-i+1})\label{good1},
\\ 
& \eta^*_i = (x - \theta^*_d)(x - \theta^*_{d-1})\cdots(x - \theta^*_{d-i+1})\label{good2},
\\
& \tau_i = (x - \theta_0)(x - \theta_{1})\cdots(x - \theta_{i-1})\label{good3},
\\ 
& \tau^*_i = (x - \theta^*_0)(x - \theta^*_{1})\cdots(x - \theta^*_{i-1})\label{good4}.
\end{align}
\end{definition}
\noindent
A straightforward algebraic calculation confirms the following.
\begin{lemma}\rm{(See \cite[Lemma $5.5$]{paramarray})}
\em
We have both
\begin{equation}\label{polysum}
    \eta_d = \sum_{i=0}^d \eta_{d-i}(\theta_0)\tau_i, \qquad \qquad \qquad \eta^*_d = \sum_{i=0}^d \eta^*_{d-i}(\theta^*_0)\tau^*_i.
\end{equation}
\end{lemma}

\noindent
Next we have some comments about linear maps. Let $A:V\to V$ be a diagonalizable linear map with eigenvalues $\theta_0, \theta_1, ..., \theta_d$. For $0 \leq i \leq d$, consider the \textit{primitive idempotent}
\begin{equation}\label{primid}E_i = \prod_{j \neq i} \frac{A - \theta_jI}{\theta_i - \theta_j}.\end{equation}
It is clear from (\ref{primid}) that $E_i$ acts as the identity on the $\theta_i$-eigenspace of $A$ and acts as the zero map on the other eigenspaces of $A$. From this observation, $E_i$ is the projection map onto the $\theta_i$-eigenspace so that 
\begin{equation}\nonumber
E_iV = \{v \in V : Av = \theta_iv\}.\end{equation}
The following relations are immediate:
\begin{align}
& \sum_{i = 0}^d E_i = I,
\\ 
& E_iE_j = \delta_{ij}E_i \qquad\qquad (0 \leq i, j \leq d),
\\
& E_iA = AE_i = \theta_iE_i \qquad (0 \leq i \leq d) \label{comm}.
\end{align}

\begin{definition}\label{parsys}
\rm
A \textit{parallel system} on $V$ is a sequence
\begin{equation}
\nonumber
    \Phi = (A; \{E_i\}^d_{i=0}; A^*; \{E^*_i\}^d_{i=0})
\end{equation}
satisfying the following conditions:
\begin{itemize}
    \item[\rm (i)] $A$ and $A^*$ are diagonalizable linear maps from $V$ to itself.
    \item[\rm (ii)] $\{E_{i}\}_{i=0}^d$ is an ordering of the primitive idempotents of $A.$
    \item[\rm (iii)] $\{E^*_{i}\}_{i=0}^{d}$ is an ordering of the primitive idempotents of $A^*.$
\end{itemize}
\end{definition}
\noindent We fix a parallel system $\Phi = (A; \{E_i\}^d_{i=0}; A^*; \{E^*_i\}^d_{i=0})$ on $V$ for the remainder of the section. For $0 \leq i \leq d$, let $\theta_i$ (resp. $\theta^*_i$) denote the eigenvalue of $A$ (resp. $A^*$) corresponding to $E_i$ (resp. $E^*_i$). 
From (\ref{primid}), we obtain
\begin{align}
    & E_0 = \frac{\eta_d(A)}{\eta_d(\theta_0)}, \qquad  \qquad \qquad E^*_0 = \frac{\eta^*_d(A)}{\eta^*_d(\theta^*_0)} \label{nstrpol},
    \\
    & E_d = \frac{\tau_d(A)}{\tau_d(\theta_d)}, \qquad  \qquad \qquad E^*_d = \frac{\tau^*_d(A)}{\tau^*_d(\theta^*_d)} \label{strpol}.
\end{align}

\begin{definition}\label{sharps}
\rm
 The parallel system $\Phi$ is said to be \textit{sharp} whenever $\dim(E^*_0V) = 1.$
\end{definition}

\noindent For the remainder of this section, assume $\Phi$ is sharp. Fix an integer $i$ with $0 \leq i \leq d$ and consider the map
\begin{equation}\nonumber
E^*_0\tau_i(A)\end{equation} on $E^*_0V$. Note that this is a linear map from $E^*_0V$ to itself. Since $E^*_0V$ has dimension $1$, it follows that this map acts as multiplication by a scalar $\chi_i$.

\begin{lemma}\label{eqtr}
Let $E:V \to V$ and $F:V \to V$ denote two linear maps such that $E^2 = E$ and \text{\rm dim}$(EV) = 1$. Then 
\begin{enumerate}
    \item [\rm (i)] $EFE = c E$, where $c = \text{\rm tr}(FE)$.
    
    \item [\rm (ii)] $\text{\rm tr}(FE)$ is nonzero if and only if $EFE$ is nonzero.
\end{enumerate}
\end{lemma}
\begin{proof}
Consider the restriction of the map $EF$ on $EV$, and observe that it must act by scalar multiplication since $\dim(EV) = 1$. Let this scalar be $c$, so that $EFE = cE$. Using $\dim(EV)=1$ and linear algebra
we obtain tr$(E)=1$. Then observe that $\text{tr}(FE) = \text{tr}(FEE) = \text{tr}(EFE) = c \text{tr}(E) = c$ by commutativity of the trace, establishing (i). The proof of (ii) follows immediately from (i).
\end{proof}

\begin{corollary}
For $0 \leq i \leq d$,
\begin{equation}\label{ned}
    \chi_i = \text{\rm{tr}}(\tau_i(A)E^*_0).
\end{equation}
\end{corollary}
\begin{proof}
Observe that $E^*_0\tau_i(A)E^*_0 = \chi_i E^*_0$. By Lemma \ref{eqtr}, $E^*_0\tau_i(A)E^*_0 = \text{\rm tr}(\tau_i(A)E^*_0)E^*_0$ as well, so it follows that $\chi_i = \text{\rm tr}(\tau_i(A)E^*_0)$. \end{proof}

\begin{definition}\label{thesplit}
\rm
For $0 \leq i \leq d$, let
\begin{equation}\nonumber
\zeta_i = (\theta^*_0 - \theta^*_1)(\theta^*_0 - \theta^*_2)\cdots(\theta^*_0 - \theta^*_i)\chi_i.\end{equation}
We call the sequence $\{\zeta_i\}_{i=0}^d$ the \textit{split sequence} of $\Phi.$ Observe that $\zeta_0 = 1.$
\end{definition}

\begin{lemma}\label{mtdtr}
The split sequence $\{\zeta_i\}_{i = 0}^d$ of $\Phi$ satisfies
\begin{align}
& \zeta_d = \eta^*_d(\theta^*_0)\tau_d(\theta_d){\rm{tr}}(E_dE^*_0) \label{big1},
\\
& \sum_{i = 0}^d \eta_{d-i}(\theta_0)\eta^*_{d-i}(\theta^*_0)\zeta_i = \eta^*_d(\theta^*_0)\eta_d(\theta_0){\rm{tr}}(E_0E^*_0) \label{big2}.
\end{align}
\end{lemma}
\begin{proof}
To obtain (\ref{big1}), observe that $\zeta_d = (\theta^*_0 - \theta^*_1)(\theta^*_0 - \theta^*_2)$ $\cdots$ $(\theta^*_0 - \theta^*_d)\chi_d$ $=$ $\eta^*(\theta^*_0)\chi_d$. From (\ref{ned}), we have $\chi_d = \text{\rm{tr}}(\tau_d(A)E^*_0)$, so (\ref{strpol}) yields (\ref{big1}), as desired.
\\
\noindent To obtain (\ref{big2}), from (\ref{nstrpol}) we have that
$\eta_d(A) = {\eta_d(\theta_0)}E_0,$
so 
\begin{equation}\nonumber
\eta^*_d(\theta^*_0)\eta_d(\theta_0){\rm{tr}}(E_0E^*_0) = \eta^*_d(\theta^*_0){\rm{tr}}(\eta_d(A)E^*_0).
\end{equation}
Now, from the equation on the left in (\ref{polysum}), we can expand the previous expression into
\begin{equation}\nonumber
    \sum_{i=0}^d \eta^*_d(\theta^*_0)\eta_{d-i}(\theta_0)\text{\rm{tr}}(\tau_i(A)E^*_0) = \sum_{i=0}^d \eta^*_d(\theta^*_0)\eta_{d-i}(\theta_0)\chi_i,
\end{equation}
 using (\ref{ned}) in the last simplification. By splitting $\eta^*_d(\theta^*_0) = \eta^*_{d-i}(\theta^*_0) (\theta^*_0 - \theta^*_1)(\theta^*_0 - \theta^*_2)\cdots(\theta^*_0 - \theta^*_i)$ and using the definition of the split sequence, we obtain (\ref{big2}).
\end{proof}

\begin{definition}
\rm
The \textit{parameter array} for $\Phi$ is the sequence $(\{\theta_i\}_{i=0}^d, \{\theta^*_i\}_{i=0}^d,$ $ \{\zeta_i\}_{i=0}^d)$.
\end{definition}

\section{Tridiagonal Systems}\label{tdsyssss}
We now recall an object known as a tridiagonal system that conveniently packages a tridiagonal pair along with a standard ordering of its eigenspaces.
\begin{definition}\rm{(See \cite[Definition~2.1]{pqassoc})}\label{TDsys}
\rm
A \textit{tridiagonal system} on $V$ is a sequence 
\begin{equation}\nonumber\Phi = (A; \{E_{i}\}_{i=0}^d; A^*; \{E^*_{i}\}_{i=0}^{d})\end{equation}
satisfying the following conditions:
\begin{itemize}
    \item[\rm (i)] $A$ and $A^*$ are diagonalizable linear maps from $V$ to itself.
    \item[\rm (ii)] $\{E_{i}\}_{i=0}^d$ is an ordering of the primitive idempotents of $A.$
    \item[\rm (iii)] $\{E^*_{i}\}_{i=0}^{d}$ is an ordering of the primitive idempotents of $A^*.$ 
    \item[\rm (iv)] $E_iA^*E_j = 0$ if $|i - j| > 1$ and $0 \leq i, j \leq d$.
    \item[\rm (v)] $E^*_iAE^*_j = 0$ if $|i - j| > 1$ and $0 \leq i, j \leq d.$
    \item[\rm (vi)] If a subspace $W$ of $V$ is such that $AW \subseteq W$ and $A^*W \subseteq W$, then either $W = 0$ or $W = V$.
\end{itemize}

\end{definition}

\begin{remark}\rm\label{rela}
From the definition above, it is straightforward to observe that if $(A; \{E_{i}\}_{i=0}^d;$ $A^*; \{E^*_{i}\}_{i=0}^{d})$ is a tridiagonal system, then so are $(A; \{E_{d-i}\}_{i=0}^d; A^*; \{E^*_{i}\}_{i=0}^{d})$, $(A; \{E_{i}\}_{i=0}^d;$ $A^*;$ $\{E^*_{d - i}\}_{i=0}^{d})$, and $(A; \{E_{d-i}\}_{i=0}^d; A^*; \{E^*_{d - i}\}_{i=0}^{d})$. These are called \textit{relatives} of the original tridiagonal system. 
\end{remark}

\begin{definition}
\rm
Referring to Definition \ref{TDsys}, we say $d$ is the \textit{diameter} of the tridiagonal system $\Phi.$
\end{definition}

\begin{flushleft}
The resemblance between Definitions \ref{TD} and \ref{TDsys} indicates the close relation between the notions of tridiagonal pairs and tridiagonal systems, as will be described in the following lemma.
\end{flushleft}

\begin{lemma}\rm{(See \cite[Lemma $2.2$ and Lemma $2.3$]{pqassoc})}\label{dual}
\em
Let $\Phi = (A; \{E_{i}\}_{i=0}^d; A^*; \{E^*_{i}\}_{i=0}^{d})$ denote a tridiagonal system on $V$. Then $(A, A^*)$ is a tridiagonal pair on $V.$ Conversely, if $(A, A^*)$ is a tridiagonal pair on $V$, then $(A; \{E_{i}\}_{i=0}^d; A^*;$ $\{E^*_{i}\}_{i=0}^{d})$ is a tridiagonal system, where $\{E_iV\}_{i=0}^d$ is a standard ordering of the $A$-eigenspaces and $\{E^*_iV\}_{i=0}^{d}$ is a standard ordering of the $A^*$-eigenspaces.
\end{lemma}

\begin{definition} 
\rm
We say a tridiagonal system $\Phi = (A; \{E_{i}\}_{i=0}^d; A^*; \{E^*_{i}\}_{i=0}^{d})$ has \textit{$q$-Serre type} if the tridiagonal pair $(A, A^*)$ has $q$-Serre type.
\end{definition}

\noindent The following is a consequence of the assumption that $\mathcal{K}$ is algebraically closed. 

\begin{proposition}\rm{(See \cite[Theorem 1.3]{struct})}\label{dm1}
\em
Suppose $(A; \{E_{i}\}_{i=0}^d; A^*; \{E^*_{i}\}_{i=0}^{d})$ is a tridiagonal system on $V$. Then $\dim(E_0V) =  \dim(E^*_0V) = 1.$
\end{proposition}

\section{The Split Decomposition}\label{spltdecompo}
Throughout this section we fix a tridiagonal system $\Phi = (A; \{E_{i}\}_{i=0}^d; A^*; \{E^*_{i}\}_{i=0}^{d})$ on $V$ with eigenvalue sequence $\{\theta_i\}_{i=0}^d$ and dual eigenvalue sequence $\{\theta^*_i\}_{i=0}^d.$ We shall be discussing decompositions of $V$. By a \textit{decomposition} of $V$, we mean a sequence of nonzero subspaces whose direct sum is $V$. For example, both $\{E_iV\}_{i=0}^d$ and $\{E^*_iV\}_{i=0}^d$ are decompositions of $V$. Another decomposition of interest to us is called the split decomposition and is described as follows.

\begin{definition}\label{splt}
\rm
A decomposition $\{U_i\}_{i=0}^d$ of $V$ is said to be \textit{{split} with respect to $\Phi$} whenever
\begin{itemize}
    \item[\rm (i)] $(A - \theta_i)U_i \subseteq U_{i+1}$ for $0 \leq i \leq d$ and $U_{d+1} = 0$,
    \item[\rm (ii)] $(A^* - \theta^*_i)U_i \subseteq U_{i-1}$ for $0 \leq i \leq d$ and $U_{-1} = 0.$
\end{itemize}

\end{definition}

\begin{flushleft}
The existence of a unique split decomposition with respect to $\Phi$ is confirmed by the following proposition.
\end{flushleft}

\begin{proposition}\rm{(See \cite[Theorem $4.6$]{pqassoc})}\label{skrt}
\em
Let $U_0, U_1, ..., U_d$ denote any subspaces of $V$. Then the following are equivalent:
\begin{itemize}
    \item[\rm (i)] $U_i = (E^*_0V + \cdots + E^*_iV) \cap (E_iV + \cdots + E_dV)$ for $0 \leq i \leq d.$
    \item[\rm (ii)] $\{U_i\}_{i=0}^d$ is a decomposition of $V$ that is split with respect to $\Phi$.
    \item[\rm (iii)] For $0 \leq i \leq d,$
    \begin{align*}
    & U_i + U_{i+1} + \cdots + U_d = E_iV + E_{i+1}V + \cdots E_dV, \label{eqsub}
    \\
    & U_0 + U_1 + \cdots + U_i = E^*_0V + E^*_1V + \cdots + E^*_iV.
    \end{align*}
\end{itemize}
\end{proposition}
\begin{proposition}\rm{(See \cite[Corollary $5.7$]{pqassoc})}\label{dmeq}
\em
Let $\{U_i\}_{i=0}^d$ be the decomposition of $V$ that is split with respect to $\Phi.$ Then 
\begin{equation}\nonumber\dim(E_iV) = \dim(E^*_iV) = \dim(U_i)\end{equation}
for $0 \leq i \leq d.$
\end{proposition}

\begin{lemma}
Let $\{U_i\}_{i=0}^d$ be the decomposition of $V$ split with respect to $\Phi.$ Then for $0 \leq i \leq d,$ we have
\begin{equation}\label{ladderincA}
    (A-\theta_{i-1}I)\cdots(A-\theta_1I)(A-\theta_0I)U_0 \subseteq U_i 
\end{equation}
and
\begin{equation}
\label{ladderincA*}
    (A^*-\theta^*_1I)(A^*-\theta^*_2I)\cdots(A^*-\theta^*_iI)U_i
    \subseteq U_0.
\end{equation}
\begin{proof}
(\ref{ladderincA}) and (\ref{ladderincA*}) follow after repeated application of inclusions (\rm{i}) and (\rm{ii}) respectively from Definition \ref{splt}.
\end{proof}
\end{lemma}
\begin{flushleft}
Let $\{U_i\}_{i=0}^d$ denote the decomposition of $V$ that is split with respect to $\Phi$. By definition of a tridiagonal system and Proposition \ref{dm1}, $\Phi$ is a sharp parallel system, so we may refer to its associated split sequence. We now use the decomposition $\{U_i\}_{i=0}^d$ to interpret this split sequence.

\noindent By Proposition \ref{skrt}, we have $U_0 = E^*_0V$. From (\ref{ladderincA}) and (\ref{ladderincA*}), the subspace $U_0$ is invariant under the map
\begin{equation}(A^* - \theta^*_1I)(A^* - \theta^*_2I)\cdots(A^* - \theta^*_{i}I)(A - \theta_{i-1}I)\cdots(A - \theta_1I)(A - \theta_0I).\label{lad}\end{equation}
Since dim$(U_0) = 1$, it follows that (\ref{lad}) acts on $U_0$ as a scalar multiple of the identity. Next, we determine this scalar. 
 
\begin{lemma}\label{sam}
Let $\{\zeta_i\}_{i=0}^d$ denote the split sequence of $\Phi.$ Then for $0 \leq i \leq d$, $\zeta_i$ is the eigenvalue for the map
\begin{equation}
    (A^* - \theta^*_1I)(A^* - \theta^*_2I)\cdots(A^* - \theta^*_{i}I)(A - \theta_{i-1}I)\cdots(A - \theta_1I)(A - \theta_0I) \label{neweq}
\end{equation}
acting on $U_0.$
\end{lemma}

\begin{proof}
We must show that the map (\ref{neweq}) and the map \begin{equation}(\theta^*_0 - \theta^*_1)(\theta^*_0 - \theta^*_2)\cdots(\theta^*_0 - \theta^*_i)E^*_0(A - \theta_{i-1}I)\cdots(A - \theta_1I)(A - \theta_0I) \label{newereq}\end{equation}
have the same action on $U_0 = E^*_0V.$ However, since $A^*E^*_0 =  \theta^*_0E^*_0$, we have
\begin{equation}\nonumber
    (A^* - \theta^*_1I)(A^* - \theta^*_2I)\cdots(A^* - \theta^*_{i}I)E^*_0 = (\theta^*_0 - \theta^*_1)(\theta^*_0 - \theta^*_2)\cdots(\theta^*_0 - \theta^*_i)E^*_0,
\end{equation}
so replacing the right hand side of this equality with the left hand side in (\ref{newereq}), we can rewrite (\ref{newereq}) as
\begin{equation}\nonumber(A^* - \theta^*_1I)(A^* - \theta^*_2I)\cdots(A^* - \theta^*_{i}I)E^*_0(A - \theta_{i-1}I)\cdots(A - \theta_1I)(A - \theta_0I).\end{equation}
At the same time, (\ref{comm}) tells us that $E^*_0$ commutes with $A^*$, so we may shift $E^*_0$ to the beginning of the product to obtain
\begin{equation}E^*_0(A^* - \theta^*_1I)(A^* - \theta^*_2I)\cdots(A^* - \theta^*_{i}I)(A - \theta_{i-1}I)\cdots(A - \theta_1I)(A - \theta_0I).\label{shft}\end{equation}
But $U_0 = E^*_0V$ is invariant under (\ref{neweq}) while $E^*_0$ is the identity on $E^*_0V$, so (\ref{shft}) has the same action as (\ref{neweq}) on $U_0$, as desired. 
\end{proof}

We mention some inequalities involving the split sequence of $\Phi$.

\begin{proposition}\rm{(See \cite[Lemma 6.1]{sharp})}\label{nzertr}
\em
We have both
\begin{align}
    & \text{\rm tr}(E_dE^*_0) \neq 0 \label{nz1},
    \\
    & \text{\rm tr}(E_0E_0^*) \neq 0 \label{nz2}.
\end{align}
\end{proposition}

\begin{corollary}\label{nozer}\rm{(See \cite[Corollary $8.3$]{sharp})}
\em
The split sequence $\{\zeta_i\}_{i = 0}^d$ of $\Phi$ satisfies
\begin{align}
& \zeta_d \neq 0 \label{nz3},
\\
& \sum_{i = 0}^d \eta_{d-i}(\theta_0)\eta^*_{d-i}(\theta^*_0)\zeta_i \neq 0 \label{nz4}.\end{align}

\end{corollary}
\begin{proof}
Observe that line (\ref{nz3}) follows from (\ref{big1}) and (\ref{nz1}), while line (\ref{nz4}) follows from (\ref{big2}) and (\ref{nz2}).
\end{proof}

As described in the following proposition, $\Phi$ is determined up to isomorphism by its parameter array. We refer the reader to \rm{\cite[Definition $5.1$]{struct}} for the definition of an isomorphism of tridiagonal systems.
\begin{proposition}\rm{(See \cite[Theorem 1.6]{struct})}\label{same}
\em
Two tridiagonal systems over $\mathcal{K}$ are isomorphic if and only if they have the same parameter array.
\end{proposition}
\end{flushleft}

\section{Mock Tridiagonal Systems}\label{mtdsysysys}
To complete our preliminary discussion, we will need the notion of a mock tridiagonal system, which is obtained from the notion of a tridiagonal system by weakening the conditions in a mild way.
\begin{definition}\rm{(See \cite[Definition~1.4]{mtdsys})}\label{MTDsys}
\rm
A \textit{mock tridiagonal system} on $V$ is a sequence 
\begin{equation}\nonumber\Phi = (A; \{E_{i}\}_{i=0}^d; A^*; \{E^*_{i}\}_{i=0}^{d})\end{equation}
satisfying the following conditions:
\begin{itemize}
    \item[\rm (i)] $A$ and $A^*$ are diagonalizable linear maps from $V$ to itself.
    \item[\rm (ii)] $\{E_i\}_{i=0}^d$ is an ordering of the primitive idempotents of $A.$
    \item[\rm (iii)] $\{E^*_i\}_{i=0}^d$ is an ordering of the primitive idempotents of $A^*.$ 
    \item[\rm (iv)] $E_iA^*E_j = 0$ if $|i - j| > 1$ and $0 \leq i, j \leq d$.
    \item[\rm (v)] $E^*_iAE^*_j = 0$ if $|i - j| > 1$ and $0 \leq i, j \leq d.$
    \item[\rm (vi)] The maps $E^*_0E_0E^*_0$ and $E^*_0E_dE^*_0$ are nonzero on $V$.
    \end{itemize}
\end{definition}

\begin{proposition}\rm{(See \cite[Lemma $1.5$]{mtdsys})}
\em
If $\Phi$ is a tridiagonal system on V, then $\Phi$ is a mock tridiagonal system on $V$.
\end{proposition}

\begin{definition}
\rm
Let $\Phi = (A; \{E_{i}\}_{i=0}^d; A^*; \{E^*_{i}\}_{i=0}^{d})$ denote a mock tridiagonal system on $V$. Then $\Phi$ is said to be \textit{sharp} if $\dim(E^*_0V) = 1.$
\end{definition}

\noindent Let $\Phi$ denote a mock tridiagonal system
that is sharp. Then $\Phi$ is also a sharp parallel system, so we may construct the associated parameter array. 

\begin{proposition}\rm{(See \cite[Theorem $2.7$ and Proposition $3.7$]{mtdsys})}\label{dult}
\em
Let $\Phi$ denote a mock tridiagonal system on $V$
that is sharp. Then there exists a vector space $V^{\ddag}$ such that $\dim(V^{\ddag}) \leq \dim(V)$ and a sharp tridiagonal system $\Phi^{\ddag}$ on $V^{\ddag}$ such that $\Phi$ and $\Phi^{\ddag}$ share the same parameter array. Moreover, if $\dim(V) = \dim(V^{\ddag})$, then $\Phi$ is a tridiagonal system isomorphic to $\Phi^{\ddag}$.
\end{proposition}

\section{Linear Perturbations of Tridiagonal Pairs}\label{linpert}
We now specialize to the setting of tridiagonal pairs of $q$-Serre type. Let $(A, A^*)$ be such a pair. From Proposition \ref{seq} there exists an eigenvalue sequence $\{q^{2i}\theta_0\}_{i=0}^d$ and a dual eigenvalue sequence $\{q^{2i}\theta^*_0\}_{i=0}^d.$ By factoring out the constants $q^d\theta_0$ and $q^{d}\theta^*_0$ respectively and utilizing Remark \ref{rela}, without loss of generality we may assume the eigenvalue sequence $\{q^{2i-d}\}_{i=0}^d$ and the dual eigenvalue sequence $\{q^{d-2i}\}_{i=0}^d$. From Lemma \ref{dual}, there exists an associated tridiagonal system $\Phi = (A; \{E_{i}\}_{i=0}^d; A^*; \{E^*_{i}\}_{i=0}^{d})$ such that $E_iV$ is the $A$-eigenspace corresponding to the eigenvalue $\theta_i = q^{2i-d}$ and $E^*_iV$ is the $A^*$-eigenspace corresponding to the eigenvalue $\theta^*_i = q^{d - 2i}.$ For the rest of the paper, we shall fix $\Phi$ to this tridiagonal system. To avoid trivialities, we assume that $d \geq 1.$

\begin{flushleft}
Now consider the decomposition $\{U_i\}_{i=0}^d$ of $V$ that is split with respect to $\Phi$. Let $K : V \to V$ be a linear map such that $U_i$ is the eigenspace of $K$ corresponding to the eigenvalue $q^{d-2i}$ for $0 \leq i \leq d$.

\begin{lemma}\label{yeet}
We have the relations
\begin{equation}\label{comrel}\frac{qKA - q^{-1}AK}{q - q^{-1}} = I, \qquad  \qquad \frac{qK^{-1}A^* - q^{-1}A^*K^{-1}}{q - q^{-1}} = I. \end{equation}
\end{lemma}
\begin{proof}
It suffices to show that these relations hold on each $U_i$ for $0 \leq i \leq d$. Let $i$ be given. Observe that 
\begin{equation}\label{rew1}
    \frac{qKA - q^{-1}AK}{q - q^{-1}} = \frac{qK(A-\theta_i I) - q^{-1}(A-\theta_i I)K}{q-q^{-1}} + \theta_i K.
\end{equation}
Apply each side of (\ref{rew1}) to $U_i$. Using the fact that on $U_i$, we have $K(A-\theta_i I) =$ $q^{d - 2i - 2}(A - \theta_i I)$ and $K = q^{d-2i} I,$ we find that the right hand side of (\ref{rew1}) is equal to $I$, establishing the relation on the left of (\ref{comrel}). Concerning the relation on the right of (\ref{comrel}), note that 
\begin{equation}\label{rew2}
    \frac{qK^{-1}A^* - q^{-1}A^*K^{-1}}{q - q^{-1}} = \frac{qK^{-1}(A^*-\theta^*_i I) - q^{-1}(A^*-\theta^*_i I)K^{-1}}{q-q^{-1}} + \theta^*_i K^{-1}.
\end{equation}
Applying each side of (\ref{rew2}) to $U_i$ and using the fact that on $U_i$, we have $K^{-1}(A^*-\theta^*_i I) =$ $q^{2i-2-d}(A^* - \theta^*_i I)$ and $K^{-1} = q^{2i-d} I$, we find that the right hand side of (\ref{rew2}) is equal to $I$. This establishes the relation on the right of (\ref{comrel}). 
\end{proof}

\begin{definition}\label{pert}
\rm
For a scalar $t$, the \textit{$t$-linear perturbation} of the tridiagonal pair $(A, A^*)$ with respect to $\Phi$ is the ordered pair $(B, B^*)$ such that
\begin{equation}
    B = A, \qquad \qquad B^* = tA^* + (1-t)K\label{deftp}. 
\end{equation}

\end{definition}

Our goal in this paper is to determine a necessary and sufficient condition on $t$ such that the $t$-linear perturbation of $(A, A^*)$ with respect to $\Phi$ is a tridiagonal pair. 
\end{flushleft}

\begin{lemma}\label{tscale}
Referring to Definition \ref{pert}, for $0 \leq i \leq d$ the following equation holds on $U_i$:
\begin{equation}\nonumber
    B^* - \theta^*_iI = t(A^* - \theta^*_iI).
\end{equation}
\end{lemma}
\begin{proof}
Since $K = \theta_i^*I$ on $U_i$, by the equality on the right of (\ref{deftp}) we find that on $U_i$, \[B^* - \theta_i^*I = tA^* + (1-t)\theta_i^*I - \theta_i^*I = t(A^* - \theta_i^*I),\] as desired.
\end{proof}

\begin{lemma}\label{blikea}
Referring to Definition \ref{pert}, for $0 \leq i \leq d$ we have
\begin{equation}\label{bsplit}(B^* - \theta^*_iI)U_i \subseteq U_{i-1}.\end{equation}
\end{lemma}
\begin{proof}

Use condition \rm{(ii)} from Definition \ref{splt} along with Lemma \ref{tscale}.\end{proof}

\begin{lemma}\label{pwrt}
Referring to Definition \ref{pert}, for $0 \leq i \leq d$ the following equation holds on $U_i$:
\begin{equation}\label{eq}
    (B^* - \theta^*_1I) \cdots (B^*-\theta^*_iI) =  t^i  (A^* - \theta^*_1I) \cdots (A^*-\theta^*_iI).
\end{equation}
Moreover, 
\begin{equation}\label{inc}
    (B^* - \theta^*_1I) \cdots (B^* - \theta^*_iI)U_i \subseteq U_0.
\end{equation}
\end{lemma}
\begin{proof}
The equality (\ref{eq}) follows from repeated use of Lemmas \ref{tscale} and \ref{blikea}. The inclusion (\ref{inc}) follows from repeated use of Lemma \ref{blikea}.
\end{proof}

\begin{lemma}\label{diag}
Referring to Definition \ref{pert}, the map $B^*$ is diagonalizable with eigenvalues $\{\theta^*_i\}_{i=0}^d$. Moreover, for $0 \leq i \leq d$ the dimension of the $\theta^*_i$-eigenspace of $B^*$ is $\dim U_i$.
\end{lemma}
\begin{proof}

First, we show that the map
\begin{equation}\label{minpol}\prod_{i = 0}^d (B^* - \theta_i^*I)\end{equation}
is zero on $V$. To this end, it suffices to show (\ref{minpol}) is the zero on $U_i$ for $0 \leq i \leq d$. Let $i$ be given. Using Lemma \ref{pwrt}, we obtain that 
\begin{equation}\label{ladder}(B^* - \theta_0^*I)(B^* - \theta_1^*I)\cdots(B^* - \theta_{i-1}^*I)(B^* - \theta_i^*I)\end{equation}
is zero on $U_i.$ The map (\ref{ladder}) is a factor of (\ref{minpol}), so (\ref{minpol}) is zero on $U_i.$ 

\noindent
We have shown that the map (\ref{minpol}) is zero on $V$. Let $m(x)$ denote the minimal polynomial of $B^*$. By our above comments, $m(x)$ divides (\ref{minpol}). In particular, $m(x)$ has no repeated roots, so $B^*$ is diagonalizable. It remains to show that for $0 \leq i \leq d,$ the dimension of the $\theta^*_i$-eigenspace of $B^*$ is $\dim U_i.$ We establish this as follows. For $0 \leq i \leq d$ choose a basis for $U_i$, so that the union of all these bases gives a basis of $V.$ Observe that under this basis, by (\ref{bsplit}) the matrix representation of $B^*$ is upper triangular, with the scalar $\theta^*_i$ appearing on the diagonal with multiplicity $\dim U_i$ for $0 \leq i \leq d.$ Then $\theta^*_i$ has multiplicity $\dim U_i$ as a root of the characteristic polynomial of $B^*$. The result follows.
\end{proof}

\begin{flushleft}
For the rest of the paper, let $E'_i$ denote the primitive idempotent of $B^*$ associated with the eigenvalue $\theta^*_i$ for $0 \leq i \leq d.$
\end{flushleft}

\begin{lemma}\label{dim1B}
We have the equality $U_0 = E'_0V$.
\end{lemma}

\begin{proof}
On one hand, since $(B^* - \theta^*_0I)U_0 = 0$, we have the containment $U_0 \subseteq E'_0V.$ On the other hand, Lemma \ref{diag} gives us that $\dim U_0 = \dim E'_0V$, so $U_0 = E'_0V$.
\end{proof}

\begin{lemma}\label{sharpps}
The sequence $\Phi' = (B; \{E_{i}\}_{i=0}^d; B^*; \{E'_{i}\}_{i=0}^{d})$ is a sharp parallel system.
\end{lemma}

\begin{proof}
By Definition \ref{pert} the map $B$ is diagonalizable with
primitive idempotents $\{E_i\}_{i=0}^d$. By Lemma \ref{diag} the map $B^*$ is diagonalizable with primitive idempotents
$\{E'_i\}_{i=0}^d$. Thus, by Definition \ref{parsys},  $\Phi'$ is a parallel system. Moreover, this parallel system is sharp because $\dim(E'_0V) = 1$ by Lemma \ref{dim1B}.
\end{proof}
\begin{definition}\label{tlinperttd}
\rm
Referring to Lemma \ref{sharpps}, we call
$\Phi'$ the \textit{$t$-linear perturbation} of $\Phi$.
\end{definition}

\noindent
Since $\Phi'$ is a sharp parallel system, by our discussion in Section \ref{pssp} we may refer to the split sequence $\{\zeta'_i\}_{i=0}^d$ associated with $\Phi'$.
\begin{lemma}\label{splittrelation}
The split sequences $\{\zeta_i\}_{i=0}^d$ of $\Phi$
and $\{\zeta'_i\}_{i=0}^d$ of $\Phi'$ are related by
$\zeta'_i = t^i \zeta_i$ for $0 \leq i \leq d$.
\end{lemma}
\begin{proof}
Recall from the discussion following Definition \ref{sharps} that for $0 \leq i \leq d$, there exists a scalar $\chi'_i$ such that the map
\begin{equation}\nonumber\label{mep}E'_0\tau_i(B)\end{equation} acts on $E'_0V$ as $\chi'_iI$. From Definition \ref{thesplit} we have
\begin{equation}\label{zpdef}\zeta'_i = (\theta^*_0 - \theta^*_1)(\theta^*_0 - \theta^*_2)\cdots(\theta^*_0 - \theta^*_i)\chi'_i.\end{equation} Applying formula (\ref{primid}) to $E'_0$ we find that $E'_0$ is equal to
\begin{equation}\label{fhalf}
\prod_{j = i+1}^d \frac{B^* - \theta^*_jI}{\theta^*_0 - \theta^*_j}
\end{equation}
times
\begin{equation}\label{shalf}
\prod_{j = 1}^i \frac{B^* - \theta^*_jI}{\theta^*_0 - \theta^*_j}.
\end{equation}
Recall $E'_0V = U_0$ from Lemma \ref{dim1B}. By (\ref{ladderincA}), we have that $\tau_i(B)U_0 \subseteq U_i.$
By Lemma \ref{pwrt}, the map (\ref{shalf}) sends $U_i$ into $U_0$, and the map (\ref{fhalf}) acts on $U_0$ as the identity.
Combining these observations with Lemma \ref{pwrt}, we find that the following equation holds on $U_0$:
\begin{equation}\label{reorg}
    E'_0\tau_i(B) = \frac{t^i(A^* - \theta^*_1I) \cdots (A^*-\theta^*_iI)(B - \theta_{i-1}I)\cdots(B - \theta_1I)(B - \theta_0I)}{(\theta^*_0 - \theta^*_1)(\theta^*_0 - \theta^*_2)\cdots(\theta^*_0 - \theta^*_i)}.
\end{equation}
Consider the map which is the common
value in (\ref{reorg}). By Lemma \ref{sam}, this map acts on $U_0$ as
\begin{equation}\nonumber\label{imprel} \frac{t^i\zeta_i}{(\theta^*_0 - \theta^*_1)(\theta^*_0 - \theta^*_2)\cdots(\theta^*_0 - \theta^*_i)}
\end{equation}
times the identity. By this and (\ref{zpdef}) we find that $\zeta'_i = t^i \zeta_i$ for $0 \leq i \leq d$.
\end{proof}

\begin{corollary}\label{parampert}
The parameter array of $\Phi'$ is the sequence $(\{\theta_i\}_{i=0}^d, \{\theta^*_i\}_{i=0}^d,$ $ \{t^i\zeta_i\}_{i=0}^d)$, where $(\{\theta_i\}_{i=0}^d, \{\theta^*_i\}_{i=0}^d,$ $ \{\zeta_i\}_{i=0}^d)$ is the parameter array of $\Phi$.
\end{corollary}
\begin{proof}
The result follows from Definition \ref{pert}, Lemma \ref{diag}, and Lemma \ref{splittrelation}. 
\end{proof}

\begin{lemma}\label{qser}
The maps $B$ and $B^*$ from Definition \ref{pert} satisfy the $q$-Serre relations.
\end{lemma}

\begin{proof}
We first show that 
\begin{equation}\label{og}B^3B^* - [3]_qB^2B^*B + [3]_qBB^*B^2 - B^*B^3 = 0.\end{equation}
Expanding the left hand side of (\ref{og}) using (\ref{deftp}), we obtain $t$ times
\begin{equation}\label{fcom}A^3A^* - [3]_qA^2A^*A + [3]_qAA^*A^2 - A^*A^3\end{equation}
plus $1-t$ times
\begin{equation}\label{scom}
    A^3K - [3]_qA^2KA + [3]_qAKA^2 - KA^3.
\end{equation}
The expression (\ref{fcom}) is zero since $A$ and $A^*$ satisfy the $q$-Serre relations. We can show (\ref{scom}) is zero by using the relation on the left of (\ref{comrel}) to pull all the $K$'s to the right, whence the resulting expression will be zero. We have thus verified (\ref{og}).

\noindent 
Next, we show that 
\begin{equation}\label{anqrel}B^{*3}B - [3]_qB^{*2}BB^* + [3]_qB^*BB^{*2} - BB^{*3} = 0.\end{equation}
Expanding the left hand side of (\ref{anqrel}) using (\ref{deftp}), we find that the expression is equal to $t^3$ times
\begin{align}\label{sad1}
A^{*3}A - [3]_qA^{*2}AA^* + [3]_qA^*AA^{*2} - AA^{*3}
\end{align}
plus $t^2(1-t)$ times
\begin{equation}\label{sad2}
\begin{split}
\left(A^{*2}KA + A^*KA^*A + KA^{*2}A\right)  -  [3]_q\left(A^{*2}AK + A^*KAA^* + KA^*AA^*\right) 
\\
 + [3]_q\left(A^*AA^*K + A^*AKA^*  + KAA^{*2}\right) - \left(AA^{*2}K + AA^*KA^* + AKA^{*2}\right)
\end{split}
\end{equation}
plus $t(1-t)^2$ times
\begin{equation}\label{sad3}
\begin{split}
\left(A^{*}K^2A + KA^*KA + K^2A^{*}A\right)  -  [3]_q\left(A^{*}KAK + KA^*AK + K^2AA^*\right) 
\\
 + [3]_q\left(A^*AK^2 + KAA^*K  + KAKA^{*}\right) - \left(AA^{*}K^2 + AKA^*K + AK^2A^{*}\right)
\end{split}
\end{equation}
plus $(1-t)^3$ times
\begin{equation}\label{sad4}
K^{3}A - [3]_qK^{2}AK + [3]_qKAK^{2} - AK^{3}.
\end{equation}
We now show that each of (\ref{sad1}), (\ref{sad2}), (\ref{sad3}), and (\ref{sad4}) is zero. First, observe that (\ref{sad1}) is zero since $A$ and $A^*$ satisfy the $q$-Serre relations. For (\ref{sad2}), we may use both relations in (\ref{comrel}) to pull all instances of $K$ to the right, upon which the resulting expression will be zero. A similar calculation establishes that (\ref{sad3}) is zero. Finally, for (\ref{sad4}), we may once more use the relation on the left of (\ref{comrel}) to pull all the $K$'s to the right, whence the resulting expression is zero.

\noindent
Having established (\ref{sad1}), (\ref{sad2}), (\ref{sad3}), and (\ref{sad4}) are all zero, we obtain the relation (\ref{anqrel}). It follows that $B$ and $B^*$ satisfy the $q$-Serre relations.\end{proof}
\begin{lemma}\label{ord}
We have that 
 \begin{equation}\label{inc1}B^*E_iV \subseteq E_{i-1}V + E_iV + E_{i+1}V\end{equation}
    for all $0 \leq i \leq d$, where $E_{-1} = E_{d+1} = 0,$ and
 \begin{equation}\label{inc2}BE'_iV \subseteq E'_{i-1}V + E'_iV + E'_{i+1}V\end{equation}
    for all $0 \leq i \leq d$, where $E'_{-1} = E'_{d+1} = 0.$
\end{lemma}
\begin{proof}
By Lemma \ref{qser}, $B$ and $B^*$ satisfy the $q$-Serre relations. In particular we have
\begin{equation}\nonumber B^3B^* - [3]_qB^2B^*B + [3]_qBB^*B^2 - B^*B^3 = 0.\end{equation}
Observe that for $0 \leq i \leq d$,
\begin{align*}
\begin{split}
0 = \left(B^3B^* - [3]_qB^2B^*B + [3]_qBB^*B^2 - B^*B^3\right)E_i
\\
= B^3B^*E_i - [3]_q\theta_iB^*E_i + [3]_q\theta^2_iB^*E_i - \theta^3_iB^*E_i
\\
= (B - q^2\theta_iI)(B - \theta_iI)(B - q^{-2}\theta_iI)B^*E_i.
\end{split}
\end{align*}
Therefore
\begin{equation}\nonumber B^*E_iV \subseteq E_{i-1}V + E_iV + E_{i+1}V\end{equation}
for all $0 \leq i \leq d$. We have shown the inclusion (\ref{inc1}). The inclusion (\ref{inc2}) is similarly shown.
\end{proof}
\begin{flushleft}
Let us summarize our results so far. For a given scalar $t$, consider the elements $B$ and $B^*$ as in Definition $\ref{pert}$. We wish to determine when $(B, B^*)$ is a tridiagonal pair. Lemma \ref{qser} showed $B$ and $B^*$ satisfy the $q$-Serre relations, while Lemmas \ref{diag} and \ref{ord} showed they satisfy conditions (i)--(iii) of Definition \ref{TD}. It remains to determine when $B$ and $B^*$ satisfy condition (iv) of Definition \ref{TD}. This will be done in the next section.
\end{flushleft}

\section{The Drinfel'd Polynomial}\label{dfp}
Recall the tridiagonal system $\Phi$ from above Lemma \ref{yeet}. The elements $A$ and $A^*$ from $\Phi$ form a tridiagonal pair on $V$ that has $q$-Serre type. For a given scalar $t$ consider the $t$-linear perturbation $(B, B^*)$ of $(A, A^*)$ as shown in (\ref{deftp}). Our aim is to find a necessary and sufficient condition on $t$ for $(B, B^*)$ to be a tridiagonal pair. To this end, we define the following polynomial. From now on let $x$ denote an indeterminate. Recall the scalar $q$ from below Remark \ref{q}.

\begin{definition}\rm{(See \cite[Section~3.4]{anotdrf} and  \cite[Definition~4.2]{drff})}
\rm
Given a sequence of scalars $\{\zeta_i\}^d_{i = 0}$, the corresponding \textit{Drinfel'd polynomial} is given by
\begin{equation}\label{mando}P(x) = \sum_{i = 0}^d \frac{(-1)^i  \zeta_ix^i}{([i]^!_q)^2},\end{equation}
where \begin{equation}\nonumber[i]^!_q = \prod_{n = 1}^i [n]_q.\end{equation}
\end{definition}

\noindent
The Drinfel'd polynomial provides a concise interpretation of (\ref{nz4}) for $\Phi$ that will be useful in proving our main theorem.

\begin{lemma}\label{ccond}
For a sequence of scalars $\{\zeta_i \}_{i=0}^d$ the corresponding Drinfel'd polynomial $P$ satisfies
\begin{equation}\nonumber\label{potato}
    \sum_{i=0}^d \eta_{d-i}(\theta_0) \eta^*_{d-i}(\theta^*_0) \zeta_i = (-1)^d (\lbrack d \rbrack^!_q)^2 (q-q^{-1})^{2d}P\left(\frac{1}{(q-q^{-1})^{2}}\right),
\end{equation}
where $\theta_i = q^{2i-d}$ and $\theta^*_i = q^{d-2i}$ for $0 \leq i \leq d.$ We are using the notation (\ref{good1})--(\ref{good4}).
\end{lemma}

\begin{proof}
We first compute $\eta_{d-i}(\theta_0)$ and $\eta^*_{d-i}(\theta^*_0)$ for $0 \leq i \leq d.$ Let $i$ be given. Expanding $\eta_{d-i}(\theta_0)$ in terms of $q$ using (\ref{good1}) yields
\begin{equation}\nonumber
    \eta_{d-i}(\theta_0) = (q^{-d} - q^d)(q^{-d} - q^{d-2})\cdots(q^{-d} - q^{2i+2-d}).
\end{equation}
Similarly, expanding $\eta^*_{d-i}(\theta^*_0)$ in terms of $q$ using (\ref{good2}) yields
\begin{equation}\nonumber
    \eta^*_{d-i}(\theta^*_0) = (q^d - q^{-d})(q^d - q^{-d + 2})\cdots(q^d - q^{d - 2i - 2}).
\end{equation}
Using the algebraic identity
\begin{equation}\nonumber(q^{-d} - q^{d - 2j})(q^d - q^{2j - d}) = -(q^{d-j} - q^{j-d})^2, \qquad \qquad j \in \mathbb{Z},\end{equation}
we obtain 
\begin{equation}\nonumber\eta_{d-i}(\theta_0)\eta^*_{d-i}(\theta^*_0) =  (-1)^{d-i}\left(\frac{[d]^!_q}{[i]^!_q}\right)^2 {(q - q^{-1})^{2d-2i}}.\end{equation}
The result follows in view of (\ref{mando}).
\end{proof}
\begin{corollary}\label{ycond}
Let $P$ denote the Drinfel'd polynomial associated with the split sequence $\{\zeta_i\}_{i=0}^d$ of $\Phi.$ Then $P$ satisfies
\begin{equation}\nonumber\label{condition}
    P\left(\frac{1}{(q-q^{-1})^2}\right) \neq 0.
\end{equation}
\end{corollary}
\begin{proof}
The result follows from Corollary \ref{nozer} and Lemma \ref{ccond}.
\end{proof}

\begin{lemma}\label{assocsplit}
Let $P$ and $P'$ denote the Drinfel'd polynomials associated with the split sequences $\{\zeta_i\}_{i=0}^d$ of $\Phi$ and $\{\zeta'_i\}_{i=0}^d$ of $\Phi'$ respectively. Then  \begin{equation}\nonumber
    P'(x) = P(tx).
\end{equation}
\end{lemma}
\begin{proof}
The result follows from Lemma \ref{splittrelation} and (\ref{mando}).
\end{proof}

\begin{lemma}\label{tdimp}
Assume that $\Phi'$ is a tridiagonal system. Then both \begin{equation}\label{main1}t \neq 0,\end{equation}
    \begin{equation}\label{main2}P\left(\frac{t}{(q - q^{-1})^2}\right) \neq 0,\end{equation}
where $P$ refers to the Drinfel'd polynomial associated with the split sequence $\{\zeta_i\}_{i=0}^d$ of $\Phi.$
\end{lemma}
\begin{proof}
First we show (\ref{main1}). Since $\Phi'$ is a tridiagonal system, from (\ref{nz3}) we have that $\zeta'_d \neq 0$, where $\{\zeta'_i\}_{i=0}^d$ is the split sequence of $\Phi'$. By Lemma \ref{splittrelation}, $\zeta'_d = t^d\zeta_d$, so it follows that $t \neq 0$. We have shown (\ref{main1}).

\noindent
Next we show (\ref{main2}). By Corollary \ref{ycond} we have
\begin{equation}\label{cond2}P'\left(\frac{1}{(q - q^{-1})^2}\right) \neq 0,\end{equation}
where $P'$ refers to the Drinfel'd polynomial associated with the split sequence $\{\zeta'_i\}_{i=0}^d$. By Lemma \ref{assocsplit}, we have  $P'(x) = P(tx)$. By this and (\ref{cond2}) we obtain (\ref{main2}).
We have shown (\ref{main1}) and (\ref{main2}), as desired. 
\end{proof}

\noindent
Referring to Lemma \ref{tdimp}, let us now reverse the logical direction.
\begin{lemma}\label{pertmtd}
Assume that both
\begin{equation}t \neq 0\label{main11},\end{equation}
    \begin{equation}P\left(\frac{t}{(q - q^{-1})^2}\right) \neq 0,\label{main22}\end{equation} where $P$ denotes the Drinfel'd polynomial associated with the split sequence $\{\zeta_i\}_{i=0}^d$ of $\Phi$. Then $\Phi'$ is a sharp mock tridiagonal system.
\end{lemma}

\begin{proof}
Recall $\Phi'$ is a sharp parallel system from Lemma \ref{sharpps}, so it suffices to verify that $\Phi'$ satisfies conditions {\rm (iv)--(vi)} of Definition \ref{MTDsys}. Lemma \ref{ord} implies that $\Phi'$ satisfies conditions {\rm (iv)} and {\rm (v)}. We now show that $\Phi'$ satisfies condition {\rm (vi)}.

\noindent
Let $\{\zeta'_i\}_{i=0}^d$ denote the split sequence for $\Phi'$, and let $P'$ denote the Drinfel'd polynomial associated with $\{\zeta'_i\}_{i=0}^d.$ From Lemma \ref{splittrelation} we have that $\zeta'_i = t^i \zeta_i$ for $0 \leq i \leq d.$ From $\zeta'_d = t^d\zeta_d$ and (\ref{main11}), we have that $\zeta'_d \neq 0$. Using (\ref{main22}) and Lemma \ref{assocsplit}, we obtain \begin{equation}\nonumber\label{checkpoint}P'\left(\frac{1}{(q - q^{-1})^2}\right) \neq 0.\end{equation} Moreover, Lemmas \ref{mtdtr} and \ref{ccond} gives us that $\text{tr}(E_dE'_0)$ and $\text{tr}(E_0E'_0)$ are nonzero. By Lemma \ref{eqtr}, the maps $E'_0E_0E'_0$ and $E'_0E_dE'_0$ are nonzero. Thus, $\Phi'$ satisfies condition {\rm (vi)} of Definition \ref{MTDsys}. We have shown that $\Phi'$ is a sharp mock tridiagonal system, as desired.
\end{proof}
\begin{proposition}\label{dualpert}
Assume that there exists a tridiagonal system $\Phi^\vee$
that has the same diameter, eigenvalue sequence, and
dual eigenvalue sequence as $\Phi$. Further assume that there exists a nonzero scalar $t$ such that $\zeta^\vee_i = t^i \zeta_i$ for $0 \leq i \leq d$, where $\{\zeta_i\}_{i=0}^d$ and $\{\zeta^\vee_i\}_{i=0}^d$ are the split sequences of $\Phi$ and $\Phi^\vee$ respectively. Then  the underlying vector spaces of $\Phi$ and $\Phi^\vee$ have the same dimension. Moreover, the $t$-linear perturbation of $\Phi$ is a tridiagonal system isomorphic to $\Phi^\vee$, and the $t^{-1}$-linear perturbation of $\Phi^\vee$ is a tridiagonal system isomorphic to $\Phi$.
\end{proposition}
\begin{proof}
Recall the underlying vector space $V$ for $\Phi$. Let $V^\vee$ denote the underlying vector space for $\Phi^\vee$. Let $P$ and $P^\vee$ denote the Drinfel'd polynomials associated with  $\{\zeta_i\}_{i=0}^d$ and $\{\zeta^\vee_i\}_{i=0}^d$ respectively. Since $\Phi^\vee$ is a tridiagonal system, Corollary \ref{ycond} implies that 
\begin{equation}\nonumber
    P^\vee\left(\frac{1}{(q-q^{-1})^2}\right) \neq 0.
\end{equation}
Recall from Definition \ref{tlinperttd} that $\Phi'$ is the $t$-linear perturbation of $\Phi.$ By Corollary  \ref{parampert}, $\Phi'$ and $\Phi^{\vee}$ share the same parameter array and hence the same Drinfel'd polynomial associated with the common split sequence, so by Lemma \ref{assocsplit} we obtain (\ref{main22}). Furthermore, (\ref{main11}) holds by assumption, so from Lemma \ref{pertmtd} we have that $\Phi'$ is a sharp mock tridiagonal system. By Proposition \ref{dult}, there exists a vector space $V^{\ddag}$ with $\dim{V^{\ddag}} \leq \dim{V}$ and a sharp tridiagonal system $\Phi^\ddag$ on $V^{\ddag}$ that shares the same parameter array as $\Phi'$.
The sharp tridiagonal systems $\Phi^\ddag$
and $\Phi^\vee$ have the same parameter array, so they are isomorphic by Proposition \ref{same}. Consequently $\dim{V^\ddag} = \dim{V^{\vee}}$,
so $\dim{V^\vee} \leq \dim{V}.$ Interchanging the roles of $\Phi$ and $\Phi^\vee$, we obtain that $\dim{V} \leq \dim{V^\vee}$. By these comments $\dim{V} = \dim{V^\vee}$. We have shown that $\dim{V} = \dim{V^\ddag}$. By this and Proposition \ref{dult}, we find that $\Phi'$ is a tridiagonal system isomorphic to
$\Phi^\ddag$. We mentioned earlier that $\Phi^\ddag$ is isomorphic to $\Phi^\vee$, so $\Phi'$ is isomorphic to
$\Phi^\vee$. Interchanging the roles of $\Phi$ and $\Phi^\vee$,
we find that the $t^{-1}$-linear perturbation of $\Phi^\vee$ is
a tridiagonal system isomorphic to $\Phi$.
\end{proof}
\noindent
In the statement of Proposition \ref{dualpert} we assumed that $\Phi^\vee$ exists. In the next result we find a necessary and sufficient condition for $\Phi^\vee$ to exist.
\begin{proposition}\label{pertex}
For a nonzero scalar $t$, there exists a tridiagonal system $\Phi^\vee$ that satisfies the
assumptions of Proposition \ref{dualpert} if and only if
\begin{equation} P\left(\frac{t}{(q - q^{-1})^2}\right) \neq 0. \label{main111}\end{equation}
\end{proposition}
\begin{proof}
Suppose (\ref{main111}) holds. By Lemma \ref{pertmtd} we have that $\Phi'$ is a sharp mock tridiagonal system.
By Corollary \ref{parampert},  $\Phi'$ has parameter array $(\{\theta_i\}_{i=0}^d, \{\theta^*_i\}_{i=0}^d,$ $ \{t^i\zeta_i\}_{i=0}^d)$, where $(\{\theta_i\}_{i=0}^d, \{\theta^*_i\}_{i=0}^d,$ $ \{\zeta_i\}_{i=0}^d)$ is the parameter array of $\Phi.$ By Proposition \ref{dult} there exists a tridiagonal system $\Phi^{\ddag}$ sharing the same parameter array as $\Phi'.$ Define the tridiagonal system $\Phi^\vee=\Phi^\ddag$
and note that $\Phi^\vee$ satisfies the assumptions in Proposition \ref{dualpert}. The result is now proved in one direction.

\noindent
We now consider the opposite direction. Suppose that there exists a tridiagonal system $\Phi^\vee$ that satisfies the assumptions of Proposition \ref{dualpert}. Then by Proposition \ref{dualpert}, $\Phi'$ is a tridiagonal system isomorphic to $\Phi^{\vee}$. Applying Lemma \ref{tdimp}, it follows that (\ref{main111}) holds.
\end{proof}
\noindent
The following is our main result.
\begin{theorem}\label{maintheorem}
Let $\Phi = (A; \{E_{i}\}_{i=0}^d; A^*; \{E^*_{i}\}_{i=0}^{d})$ denote a tridiagonal system on $V$ with eigenvalue sequence $\{\theta_i\}_{i=0}^d = \{q^{2i-d}\}_{i=0}^d$ and dual eigenvalue sequence $\{\theta^*_i\}_{i=0}^d = \{q^{d-2i}\}_{i=0}^d$. For a given scalar $t$ consider the maps $B$ and $B^*$ from (\ref{deftp}) and the parallel system $\Phi'$ from Lemma \ref{sharpps}. Then the following are equivalent:
\begin{enumerate}
    \item[\rm (i)] the pair $(B, B^*)$ is a tridiagonal pair on $V$;
    \item[\rm (ii)] $\Phi'$ is a tridiagonal system;
    \item[\rm (iii)] both
    \begin{equation}\nonumber t \neq 0, \qquad \qquad P\left(\frac{t}{(q - q^{-1})^2}\right) \neq 0,\end{equation}
    where $P$ refers to the Drinfel'd polynomial associated with the split sequence $\{\zeta_i\}_{i=0}^d$ of $\Phi.$
\end{enumerate}
Moreover, assume that {\rm{(i)--(iii)}} hold. Then $\Phi'$ has eigenvalue sequence $\{\theta_i\}_{i=0}^d$, 
dual eigenvalue sequence $\{\theta^*_i\}_{i=0}^d$, and split sequence $\{t^i\zeta_i\}_{i=0}^d$. In particular, $\Phi'$ and $(B, B^*)$ have $q$-Serre type.
\end{theorem}
\begin{proof}
$\rm{(i)}\Leftrightarrow \rm{(ii)}$ By Lemmas \ref{dual} and \ref{ord}.

\noindent
$\rm{(ii)}\Rightarrow \rm{(iii)}$ By Lemma \ref{tdimp}.

\noindent
$\rm{(iii)}\Rightarrow \rm{(ii)}$ By Propositions \ref{dualpert} and \ref{pertex}.

\noindent
Now assume that {\rm(i)--(iii)} hold. Then the final assertions of the theorem statement follow from Proposition \ref{seq} and Corollary \ref{parampert}.\end{proof}

\section{Acknowledgements}
The author would like to express immense gratitude to Professor Paul Terwilliger from the University of Wisconsin-Madison for his mentorship and numerous valuable suggestions.


\begin{thebibliography}{7}
\bibitem{physics}
P.~Baseilhac.
\newblock Deformed {D}olan-{G}rady relations in quantum integrable
models.
\newblock {\em
Nuclear Phys. B}
 \textbf{709}
 (2005)
 491--521;
 {\tt arXiv:hep-th/0404149}.


\bibitem{k}
S. Bockting-Conrad. Tridiagonal pairs of $q$-Racah type, the double lowering operator $\psi$, and
the quantum algebra $U_q({\mathfrak{sl}}_2)$. \textit{Linear Algebra Appl.,} \textbf{445} (2014) 256–279; {\tt arXiv:1307.7410.}

\bibitem{anotk}
S. Bockting-Conrad and P. Terwilliger. The algebra $U_q({\mathfrak{sl}}_2)$ in disguise. \textit{Linear Algebra Appl., }\textbf{459}
(2014) 548–585.


\bibitem{anotdrf}
V. Chari and A. Pressley. Quantum affine algebras. \textit{Commun. Math. Phys., }\textbf{142} (1991)
261–283.
\bibitem{pqassoc}
    T.~Ito, K.~Tanabe, and P.~Terwilliger.
        \newblock Some algebra related to $P$-and $Q$-polynomial association schemes. In \newblock \textit{Codes and Association Schemes 1999 (DIMACS)}, 167--192, Amer. Math. Soc., Providence RI, 2001; {\tt arXiv:math.QA/0406556}.
\bibitem{mtdsys}
    T.~Ito and P.~Terwilliger.
    \newblock Mock tridiagonal systems. \textit{Linear Algebra Appl.,} \textbf{435} (2011) 1997--2006; {\tt arXiv:0807.4360}.
    
\bibitem{yetanotk}
T. Ito and P. Terwilliger. The augmented tridiagonal algebra. \textit{Kyushu J. Math.,} \textbf{64} (2010) 81–144. {\tt arXiv:0904.2889}.

\bibitem{drinfeld}
    T.~Ito and P.~Terwilliger.
        \newblock The Drinfel'd polynomial of a tridiagonal pair. \textit{J. Combin.
Inform. System Sci.,} \textbf{34} (2009), 255--292; {\tt arXiv:0805.1465}.
        
\bibitem{kk}
T.~Ito and P.~Terwilliger.
    \newblock Tridiagonal pairs and the quantum affine algebra $U_q(\widehat{\mathfrak{sl}}_2)$. \textit{Ramanujan J., }\textbf{13} (2007), no. 1-3, 39–62; {\tt arXiv:math/0310042}.
\bibitem{drff}
T.~Ito and P.~Terwilliger. Two non-nilpotent linear transformations that satisfy the cubic $q$-Serre
relations. \textit{J. Algebra Appl.,} \textbf{6} (2007) 477–503; {\tt arXiv:math/0508398}.
	     

\bibitem{paramarray}
    K.~Nomura and P.~Terwilliger.
        \newblock Matrix units associated with the split basis of a Leonard pair. \textit{Linear Algebra Appl., }\textbf{418} (2006), no. 2-3, 775–787; {\tt arXiv:0602416}.
\bibitem{sharp}
    K.~Nomura and P.~Terwilliger.
        \newblock Sharp tridiagonal pairs. \textit{Linear Algebra Appl.,} \textbf{429} (2008), no. 1, 79–99; {\tt arXiv:0712.3665}.
        
        
\bibitem{struct}
    K.~Nomura and P.~Terwilliger.
        \newblock The structure of a tridiagonal pair. \textit{Linear Algebra Appl.,} \textbf{429} (2008), no. 7, 1647–1662; {\tt arXiv:0802.1096}.
        
\bibitem{qSerre}
  P.~Terwilliger.
    \newblock Two relations that generalize the $q$-Serre relations and the
      Dolan-Grady relations. In
        \newblock {\em  Physics and
	  Combinatorics 1999 (Nagoya)}, 377--398, World Scientific Publishing,
	     River Edge, NJ, 2001;
	     {\tt arXiv:math.QA/0307016}.

    

\end{thebibliography}
\end{document}